\documentclass{amsart}
\usepackage{amssymb,amstext,amsmath,amscd,amsthm,amsfonts,enumerate,graphicx,latexsym,stmaryrd,multicol}

\usepackage{comment} 

\usepackage[usenames]{color}
\usepackage[all]{xy}

\newtheorem{thm}{Theorem}[section]
\newtheorem{lem}[thm]{Lemma}
\newtheorem{prop}[thm]{Proposition}
\newtheorem{cor}[thm]{Corollary}
\theoremstyle{definition}
\newtheorem{dfn}[thm]{Definition}
\newtheorem{ques}[thm]{Question}
\newtheorem{rem}[thm]{Remark}
\newtheorem{conv}[thm]{Convention}

\newtheorem{ex}[thm]{Example}
\newtheorem{nota}[thm]{Notation}
\theoremstyle{remark}

\numberwithin{equation}{thm}
\def\add{\operatorname{\mathsf{add}}}
\def\ass{\operatorname{Ass}}

\def\C{\mathcal{C}}  
\def\cm{\mathsf{CM}}
\def\cokernel{\operatorname{Coker}} 

\def\depth{\operatorname{depth}}

\def\e{\operatorname{e}}
\def\edim{\operatorname{edim}}
\def\Ext{\operatorname{Ext}}
\def\ext{\operatorname{\mathsf{ext}}}

\def\G{\mathcal{G}}
\def\g{\mathrm{G}} 
\def\gp{\mathsf{GP}}
\def\ge{\geqslant}
\def\grade{\operatorname{grade}}

\def\height{\operatorname{ht}}
\def\Hom{\operatorname{Hom}}

\def\id{\operatorname{id}}
\def\image{\operatorname{Im}}
\def\ipd{\operatorname{IPD}} 

\def\le{\leqslant}

\def\m{\mathfrak{m}}

\def\Min{\operatorname{Min}} 
\def\mod{\operatorname{\mathsf{mod}}}

\def\nf{\operatorname{NF}}
\def\ng{\operatorname{NonGor}}

\def\p{\mathfrak{p}}
\def\pd{\operatorname{pd}}
\def\proj{\operatorname{\mathsf{proj}}}

\def\q{\mathfrak{q}}

\def\r{\operatorname{r}}
\def\rank{\operatorname{rank}}
\def\refl{\operatorname{\mathsf{ref}}}
\def\res{\operatorname{\mathsf{res}}}

\def\s{\mathsf{S}}
\def\sing{\operatorname{Sing}}
 
\def\spec{\operatorname{Spec}}

\def\ss{\mathrm{S}}
\def\supp{\operatorname{Supp}}
\def\Syz{\mathsf{Syz}}
\def\syz{\mathsf{\Omega}}

\def\tf{\mathsf{TF}}  
\def\Tor{\operatorname{Tor}}
\def\tr{\operatorname{\mathsf{Tr}}}

\def\ul{\mathsf{Ul}} 

\def\V{\operatorname{V}}

\def\X{\mathcal{X}}
\def\xx{\boldsymbol{x}}

\def\Y{\mathcal{Y}}

\def\ZZ{\mathbb{Z}}
\def\zz{\boldsymbol{z}}
\begin{document}
\allowdisplaybreaks
\title[$n$-torsionfree modules and related modules]{On the subcategories of $n$-torsionfree modules\\
and related modules}
\author{Souvik Dey}
\address[S.D.]{Department of Mathematics, University of Kansas, 1460 Jayhawk Blvd., Lawrence, KS 66045, U.S.A.}
\email{souvik@ku.edu}
\author{Ryo Takahashi}
\address[R.T.]{Graduate School of Mathematics, Nagoya University, Furocho, Chikusaku, Nagoya 464-8602, Japan}
\email{takahashi@math.nagoya-u.ac.jp}
\urladdr{https://www.math.nagoya-u.ac.jp/~takahashi/}
\thanks{The author was partly supported by JSPS Grant-in-Aid for Scientific Research 19K03443}
\begin{abstract}
Let $R$ be a commutative noetherian ring.
Denote by $\mod R$ the category of finitely generated $R$-modules.
In this paper, we study $n$-torsionfree modules in the sense of Auslander and Bridger, by comparing them with $n$-syzygy modules, and modules satisfying Serre's condition $(\ss_n)$.
We mainly investigate closedness properties of the full subcategories of $\mod R$ consisting of those modules.
\end{abstract}
\maketitle
\section{Introduction} 

The notion of $n$-torsionfree modules for $n\ge0$ has been introduced by Auslander and Bridger \cite{AB}, and actually plays an essential role in their wide and deep theory on the stable category of finitely generated modules over a noetherian ring.
The modules located at the center of the $n$-torsionfree modules are the totally reflexive modules, which are also called Gorenstein projective modules or modules of Gorenstein dimension zero.
So far a lot of studies have been done on $n$-torsionfree modules; we recommend the reader to check the more than 500 articles which cite \cite{AB}.

Let $R$ be a commutative noetherian ring.
Let us denote by $\mod R$ the category of finitely generated $R$-modules, by $\tf_n(R)$ the full subcategory of $n$-torsionfree modules, by $\Syz_n(R)$ the full subcategory of $n$-syzygy modules, and by $\s_n(R)$ the full subcategory of modules satisfying Serre's condition $(\ss_n)$.
The main purpose of this paper is to investigate the structure of the subcategory $\tf_n(R)$, and toward this we shall mainly compare $\tf_n(R)$ with $\Syz_n(R)$ and $\s_n(R)$.
In what follows, we explain our main results.
For simplicity, we assume that $R$ is a local ring with residue field $k$, and has dimension $d$ and depth $t$.

We begin with considering how to describe $\s_n(R)$ as the extension and resolving closures of $\tf_n(R)$ and $\Syz_n(R)$.
We obtain a couple of sufficient conditions for $\s_n(R)$ to coincide with the resolving closure of $\tf_n(R)$ as in the following theorem, which are included in Proposition \ref{3-2}(4) and Theorem \ref{tf-minimal}(1).

\begin{thm}\label{1-1}
It holds that $\s_n(R)$ is the smallest resolving subcategory of $\mod R$ containing $\tf_n(R)$, if the ring $R$ satisfies $(\ss_n)$ and either of the following holds.
\begin{enumerate}[\rm(1)]
\item
The ring $R$ is locally Gorenstein in codimension $n-2$.
\item
One has $n\le d+1$, and $R$ is locally a Cohen--Macaulay ring with minimal multiplicity on the punctured spectrum such that for each $\p\in\sing R$ there exists $M\in\tf_n(R)$ satisfying $\pd_{R_\p}M_\p=\infty$.
\end{enumerate}
\end{thm} 

The nature of proof of the above theorem naturally tempts us to consider the $n$-torsionfree property of syzygies of $k$, as a consequence of which we can also study when $\Syz_n(R)$ is closed under direct summands.
We obtain various equivalent conditions for $R$ to be Gorenstein in terms of $\tf_n(R)$ and $\Syz_n(R)$, which are stated in the theorem below.
This is included in Theorems \ref{syz}(3), \ref{type 1}(3), \ref{ext-local-thm} and Proposition \ref{tf-syz}(5b).

\begin{thm}\label{1-2}
The following are equivalent.
\begin{enumerate}[\rm(1)]
\item
The ring $R$ is Gorenstein.
\item
The $n$th syzygy $\syz^nk$ of the $R$-module $k$ belongs to $\tf_{t+2}(R)$ for some $n\ge t+1$.
\item
The subcategory $\tf_{t+1}(R)$ is closed under extensions.
\item
The subcategory $\Syz_n(R)$ is closed under extensions for some $n\ge t+1$.
\item
The ring $R$ is Cohen--Macaulay, and the syzygy $\syz^tk$ belongs to $\Syz_{t+2}(R)$.
\item
The ring $R$ is Cohen--Macaulay, and the subcategory $\Syz_{t+2}(R)$ is closed under direct summands.
\end{enumerate}
\end{thm}

We also obtain the following theorem, which yields sufficient conditions for $\tf_n(R)$ to coincide with the subcategory of totally reflexive modules.
This is included in Proposition \ref{tf-syz}(5) and Corollary \ref{5-8}.

\begin{thm}\label{1-3}
It holds that $\tf_n(R)$ consists of totally reflexive modules, if either of the following holds.
\begin{enumerate}[\rm(1)]
\item
One has $n\ge t+2$, and the subcategory $\tf_n(R)$ is closed under syzygies.
\item
One has $n\ge t+1$, and the subcategory $\tf_n(R)$ is closed under extensions.
\item
One has $n\ge1$, and there is an equality $\tf_n(R)=\tf_{n+1}(R)$ of subcategories.
\end{enumerate}
\end{thm}

The organization of the paper is as follows.
Section 2 is devoted to preliminaries for the later sections.
In Section 3, we try to describe $\s_n(R)$ as the resolving closure of $\tf_n(R)$.
We prove Theorem \ref{1-1} and give various other related desciptions.
We also derive some partial converses to Theorem \ref{1-1}(1).
In Sections 4 and 5, we prove a much more general version of Theorem \ref{1-2}, and Theorem \ref{1-3}.
We also consider extending the implication (6) $\Rightarrow$ (1) in Theorem \ref{1-2} for $\Syz_n(R)$ with $n>t+2$, and provide both positive and negative answers.
Finally, we investigate some other subcategories related to $\tf_n(R)$ and $\s_n(R)$.

\section{Preliminaries} 

In this section, we introduce some basic notions, notations and terminologies that will be used tacitly in the later sections of the paper.

\begin{conv}
Throughout the paper, let $R$ be a commutative noetherian ring.
All modules are assumed to be finitely generated and all subcategories be strictly full.
Subscripts and superscripts may be omitted unless there is a danger of confusion.
We may identify each object $M$ of a category $\C$ with the subcategory of $\C$ consisting just of $M$.
\end{conv} 

\begin{nota}
\begin{enumerate}[(1)]
\item
The {\em singular locus} of $R$ is denoted by $\sing R$, which is defined as the set of prime ideals $\p$ of $R$ such that $R_{\p}$ is singular (i.e., nonregular).
\item
We denote by $\mod R$ the category of (finitely generated) $R$-modules, by $\proj R$ the subcategory of $\mod R$ consisting of projective $R$-modules, and by $\cm(R)$ the subcategory of $\mod R$ consisting of maximal Cohen--Macaulay $R$-modules (recall that an $R$-module $M$ is called {\em maximal Cohen--Macaulay} if $\depth M_\p=\dim R_\p$ for all $\p\in\supp M$).
\item
We denote by $(-)^\ast$ the $R$-dual functor $\Hom_R(-,R)$ from $\mod R$ to itself.
\item
Let $M$ and $N$ be $R$-modules.
We write $M \lessdot N$ to mean that $M$ is isomorphic to a direct summand of $N$.
By $M\approx N$ we mean that $M\oplus P\cong N\oplus Q$ for some projective $R$-modules $P$ and $Q$.
\end{enumerate}
\end{nota}

\begin{dfn}
Let $M$ be an $R$-module.
\begin{enumerate}[(1)]
\item
We denote by $\syz M$ the kernel of a surjective homomorphism $P\to M$ with $P\in\proj R$ and call it the {\em first syzygy} of $M$.
Note that $\syz M$ is uniquely determined up to projective summands.
For $n\ge 1$ we inductively define the {\em $n$th syzygy} $\syz^nM$ of $M$ by $\syz^n M:=\syz(\syz^{n-1}M)$, where we put $\syz^0M=M$.
\item
Let $P_1\xrightarrow{d}P_0\to M\to 0$ be a projective presentation of $M$.
We set $\tr M=\cokernel(d^\ast)$ and call it the {\em (Auslander) transpose} of $M$.
It is uniquely determined up to projective summands; see \cite{AB} for details.
\item
Suppose that $R$ is local.
Then one can take a minimal free resolution $\cdots\xrightarrow{d_2}F_1\xrightarrow{d_1}F_0\xrightarrow{d_0}M\to0$ of $M$.
We can define $\syz^nM$ and $\tr M$ as $\image d_n$ and $\cokernel(d_1^\ast)$, respectively.
Recall that a minimal free resolution is uniquely determined up to isomorphism.
{\em Whenever $R$ is local, we define syzygies and transposes by using minimal free resolutions, so that they are uniquely determined up to isomorphism.}
\end{enumerate}
\end{dfn} 

\begin{dfn}
Let $\X$ be a subcategory of $\mod R$.
We say that $\X$ is {\em closed under extensions} ({\em closed under kernels of epimorphisms}) if for an exact sequence $0\to L\to M\to N\to 0$ in $\mod R$ with $L,N\in \X$ (resp. $M,N\in\X$) it holds that $M\in \X$ (resp. $L\in\X$).
We say that $\X$ is {\em resolving} if it contains $\proj R$ and is closed under direct summands, extensions and kernels of epimorphisms.
Note that $\X$ is resolving if and only if it contains $R$ and is closed under direct summands, extensions and syzygies (since an exact sequence $0\to L\to M\to N\to0$ induces an exact sequence $0\to\syz N\to L\oplus P\to M\to0$ with $P\in\proj R$).
\end{dfn}

\begin{dfn}
Let $\X$ be a subcategory of $\mod R$.
\begin{enumerate}[(1)]
\item
Denote by $\add\X$ (resp. $\ext\X$) the {\em additive closure} (resp. {\em extension closure}) of $\X$, that is, the smallest subcategory of $\mod R$ containing $\X$ and closed under finite direct sums and direct summands (resp. closed under direct summands and extensions).
Note that $\add R=\proj R$ and $\add\X\subseteq\ext\X\subseteq\res\X$.
\item
Denote by $\syz \X$ the subcategory of $\mod R$ consisting of $R$-modules $M$ that fits into an exact sequence $0\to M \to P \to X\to 0$ in $\mod R$ with $P\in\proj R$ and $X\in \X$.
Denote by $\tr\X$ the subcategory of $\mod R$ consisting of $R$-modules of the form $\cokernel(d^\ast)$, where $d:P_1\to P_0$ is a homomorphism of projective $R$-modules such that $\cokernel d$ belong to $\X$.
For each $n\ge0$, we inductively define $\syz^n\X$ by $\syz^0\X:=\X$ and $\syz^n\X:=\syz(\syz^{n-1}\X)$.
Note that $\proj R\subseteq\syz^n\X\cap\tr\X$.
We set $\Syz_n(R)=\syz^n(\mod R)$.
Then $\Syz_n(R)$ consists of those modules $M$ that fits into an exact sequence $0\to M\to P_{n-1}\to\cdots \to P_0$ with $P_i\in\proj R$ for each $i$.
We say that an $R$-module is {\em $n$-syzygy} if it belongs to $\Syz_n(R)$.
\end{enumerate}
\end{dfn}

\begin{dfn}
Let $n\ge0$ be an integer.
We say that $R$ satisfies $(\g_n)$ if $R_{\p}$ is Gorenstein for all prime ideals $\p$ of $R$ with $\dim R_{\p}\le n$. 
We denote by $\widetilde \s_n(R)$ (resp. $\s_n(R)$) the subcategory of $\mod R$ consisting of $R$-modules $M$ satisfying Serre's condition $(\widetilde\ss_n)$ (resp. $(\ss_n)$), that is to say, $\depth M_\p\ge\min\{n,\depth R_{\p}\}$ (resp. $\depth M_\p\ge\min\{n,\dim R_{\p}\}$) for all prime ideals $\p$ of $R$.
By the depth lemma $\widetilde \s_n(R)$ is a resolving subcategory of $\mod R$ containing $\Syz_n(R)$, and $\s_n(R)=\widetilde \s_n(R)$ holds if (and only if) $R$ satisfies $(\ss_n)$.
\end{dfn} 

\begin{dfn}
Let $M$ be an $R$-module, $\X$ a subcategory of $\mod R$ and $\Phi$ a subset of $\spec R$.
\begin{enumerate}[(1)]
\item
We denote by $\ipd(M)$ the {\em infinite projective dimension locus} of $M$, that is, the set of prime ideals $\p$ of $R$ with $\pd_{R_\p}M_\p=\infty$.
We set $\ipd(\X)=\bigcup_{X\in\X}\ipd(X)$.
We denote by $\ipd^{-1}(\Phi)$ the subcategory of $\mod R$ consisting of modules $M$ with $\ipd(M)\subseteq\Phi$.
Note that $\ipd^{-1}(\Phi)$ is a resolving subcategory of $\mod R$ and $\ipd(\res\X)=\ipd(\X)\subseteq \sing R$.
Also, $\ipd (\Syz_n(R))=\sing R$ for any $n\ge 0$, as $\p\in\ipd(\syz^n_R(R/\p))$ for $\p\in\sing R$.
If $R$ satisfies $(\ss_n)$, then $\ipd (\ss_n(R))=\sing R$, since $\Syz_n(R)\subseteq \s_n(R)$.
\item
The {\em nonfree locus} $\nf(M)$ of $M$ is defined as the collection of all prime ideals $\p$ of $R$ such that $M_{\p}$ is not $R_{\p}$-free.
We put $\nf(\X)=\bigcup_{M\in \X}\nf(M)$.
It holds that $\ipd(\X)\subseteq\nf(\X)=\nf (\res \X)$.
If $R$ is Cohen--Macaulay, then $\s_n(R)=\cm(R)$ for every $n\ge\dim R$ and $\ipd (\X)=\nf(\X)$ if $\X\subseteq \cm (R)$, whence there are equalities $\ipd(\cm(R))=\nf(\cm(R))=\sing R$.
\end{enumerate}
\end{dfn}

\begin{dfn}
Let $a,b\in\ZZ_{\ge0}\cup\{\infty\}$.
By $\G_{a,b}$ we denote the subcategory of $\mod R$ consisting of $R$-modules $M$ such that $\Ext^i_R(M,R)=0=\Ext^j_R(\tr M,R)$ for all $1\le i\le a$ and $1\le j\le b$.
By definition, the $R$-modules in $\G_{\infty,\infty}$ (resp. $\G_{0,n}$ for $n\ge0$) are the {\em totally reflexive} (resp. {\em $n$-torsionfree}) $R$-modules.
We put $\gp(R)=\G_{\infty,\infty}$ and $\tf_n(R)=\G_{0,n}$.
Note that there are equalities $\tf_1(R)=\Syz_1(R)$ and $\tf_2(R)=\refl R$, where $\refl R$ stands for the subcategory of $\mod R$ consisting of reflexive $R$-modules.
\end{dfn}  

\section{Representing $\s_n(R)$ as the resolving closure of $\tf_n(R)$} 

In this section we represent the subcategories $\s_n(R)$ and $\widetilde\s_n(R)$ as the extension and resolving closures of $\tf_n(R)$ and $\Syz_n(R)$.
We begin with establishing a lemma.

\begin{lem}\label{3-1}
Let $\X$ be a subcategory of $\mod R$.
\begin{enumerate}[\rm(1)]
\item
If $\X$ contains $R$ and is closed under syzygies, then there is an equality $\ext\X=\res\X$.
\item
If $\X$ is resolving, then $\syz \X$ is closed under syzygies and direct summands.
\end{enumerate}
\end{lem}

\begin{proof}
(1) It suffices to show that $\ext\X$ is closed under syzygies.
Consider the subcategory $\Y$ of $\mod R$ consisting of modules $M$ such that $\syz M\in\ext\syz\X$. 
It is easy to observe that $\Y$ contains $\X$ and is closed under direct summands and extensions.
We obtain $\syz(\ext\X)\subseteq\ext\syz\X\subseteq\ext\X$.

(2) Since $\syz \X\subseteq \X$, we have $\syz(\syz\X)\subseteq\syz\X$.
Let $0\to M\oplus N\to F\to X\to0$ be an exact sequence with $F\in\add R$ and $X\in\X$.
Then by \cite[Lemma 3.1]{syz2} we get exact sequences $0\to M\to F\to A\to0$ and $0\to F\to A\oplus B\to X\to0$.
As $\X$ is resolving, the latter exact sequence shows $A\in\X$, and then from the former we obtain $M\in\syz\X$.
Therefore, $\syz\X$ is closed under direct summands.
\end{proof} 

Let $R$ be a Cohen--Macaulay local ring.
We say that $R$ has {\em minimal multiplicity} if the equality $\e(R)=\edim R-\dim R+1$ holds.
A maximal Cohen--Macaulay $R$-module $M$ is called {\em Ulrich} if $\e(M)=\mu(M)$.
We denote by $\ul(R)$ the subcategory of $\cm(R)$ consisting of Ulrich $R$-modules.
In the proposition below we provide several descriptions as extension and resolving closures.

\begin{prop}\label{3-2}
\begin{enumerate}[\rm(1)]
\item
There are equalities $\widetilde \s_n(R)=\ext\Syz_n(R)=\res\Syz_n(R)$.
\item
If $R$ satisfies $(\ss_n)$, then one has $\s_n(R)=\ext\Syz_n(R)=\res\Syz_n(R)$.
The converse holds as well.
\item
If $R$ is Cohen--Macaulay, then $\cm(R)=\ext \Syz_n(R)=\res\Syz_n(R)$ for all $n\ge\dim R$.
\item
If $R$ satisfies $(\ss_n)$ and $(\g_{n-2})$, then the equalities $\s_n(R)=\ext \tf_{n}(R)=\res\tf_{n}(R)$ hold.
\item
If $R$ is a local Cohen--Macaulay ring of minimal multiplicity, then $\cm(R)=\ext\ul(R)=\res\ul(R)$.
\end{enumerate}
\end{prop}

\begin{proof}
(1) By Lemma \ref{3-1}(1), we have $\X:=\ext\Syz_n(R)=\res\Syz_n(R)\subseteq\widetilde\s_n(R)$.
For each $\p\in\spec R$ the module $\syz^n(R/\p)$ belongs to $\X$.
Hence $\X$ is dominant in the sense of \cite{crspd}.
Fix $M\in \widetilde \s_n(R)$ and $\p\in\spec R$, so that $\depth M_\p\ge\min\{n,\depth R_\p\}$.
In view of \cite[Theorem 1.1]{dom}, it suffices to show that there exists $X\in\X$ with $\depth M_\p\ge\depth X_\p$.
If $\depth R_{\p} \le n$, then $\depth M_\p\ge\depth R_{\p}$, and we are done since $R\in\X$. 
Now suppose $\depth R_{\p} >n$. 
Then $\depth M_\p\ge n < \depth R_\p$.
Setting $X=\syz^n(R/\p)$, we have $X\in\Syz_n(R)\subseteq\X$ and $X_\p\cong\syz^n\kappa(\p)\oplus R_\p^{\oplus a}$ for some $a\ge0$.
As $\depth R_\p > n$ and $\depth \kappa(\p)=0$, we have $\depth \syz^n \kappa(\p)=n$, and so we get
$$
\depth X_\p=
\begin{cases}
\depth\syz^n\kappa(\p)=n & \text{if }a=0,\\
\inf\{\depth R_\p,\,\depth\syz^n\kappa(\p)\}
=\inf\{\depth R_\p,\,n\}=n & \text{if }a>0.
\end{cases}
$$
Thus $\depth M_\p\ge n=\depth X_\p$, and the proof is completed. 

(2) The first assertion follows from (1).
The second assertion holds since we have $R\in\Syz_n(R)\subseteq\ext\Syz_n(R)=\s_n(R)$.

(3) We have $R\in\s_n(R)$ for $n\ge0$ and $\cm(R)=\s_n(R)$ for $n\ge \dim R$.
The assertion follows by (2).

(4) We may assume $n\ge1$.
By \cite[Theorem 2.3(2)$\Rightarrow$(6)]{MTT} we have $\tf_n(R)=\Syz_n(R)$.
Apply (2).

(5) Put $d=\dim R$.
We have $\ext\ul(R)\subseteq\res\ul(R)\subseteq\cm(R)=\ext\Syz_{d+1}(R)\subseteq\ext\syz\cm(R)$, where the equality follows from (3).
It thus suffices to show $\syz\cm(R)\subseteq\ext\ul(R)$, and for this we may assume that $R$ is singular.
Let $k$ be the residue field of $R$.
Take an exact sequence $0\to \syz^{d+2} k \to F \to \syz^{d+1} k \to 0$ with $F$ nonzero and free.
The modules $\syz^{d+1}k,\syz^{d+2}k\in\syz\cm(R)$ have no nonzero free summands by \cite[Corollary 1.3]{D}.
It follows from \cite[Proposition 3.6]{umm} that $\syz^{d+1}k,\syz^{d+2}k$ are in $\ext\ul(R)$, so is $F$, and so is $R\lessdot F$.
We obtain $\syz\cm(R)\subseteq\ext\ul(R)$ by \cite[Proposition 3.6]{umm} again.
\end{proof}

In Proposition \ref{3-2}(4) we got a description of $\s_n(R)$ as the resolving closure of $\tf_n(R)$ under a certain assumption on the Gorenstein locus of $R$.
We have the same description regardless of the Gorenstein locus in the theorem below.

\begin{thm}\label{tf-minimal}  
Let $(R,\m,k)$ be a local ring of dimension $d$.
\begin{enumerate}[\rm(1)]
\item
Suppose that $R$ is locally a Cohen--Macaulay ring with minimal multiplicity on $\spec_0R$. 
Let $0\le n\le d+1$ be such that $R$ satisfies $(\ss_n)$.
Then $\sing R=\ipd(\tf_n(R))$ if and only if $\s_n(R)=\res\tf_n(R)$.
\item
Assume that $R$ is Cohen--Macaulay ring and locally has minimal multiplicity on $\spec_0R$.
Then $\sing R=\nf(\tf_n(R))$ if and only if $\cm(R)=\res\tf_n(R)$ for $n=d,d+1$.
\end{enumerate}
\end{thm}

\begin{proof}
(2) The assertion is immediate from (1).

(1) The ``if'' part is obvious.
In what follows, we show the ``only if'' part.
We may assume $n\ge1$.
Set $\X=\res\tf_n(R)$.
We have $\tf_n(R)\subseteq\Syz_n(R)\subseteq\s_n(R)$ and $\s_n(R)$ is resolving.
Thus it is enough to show that $\X$ contains $\s_n(R)$. 

Put $t=\depth R$.
As $R$ satisfies $(\ss_n)$, we have $t\ge\inf\{n,d\}$.
As $n\le d+1$ by assumption, we get
\begin{equation}\label{7-1}
1\le n\le t+1.
\end{equation}
The module $\Ext_R^i(k,R)$ has grade at least $i-1$ for each $1\le i\le t+1$.
By \cite[Proposition (2.26)]{AB}, the module $\syz^ik$ is $i$-torsionfree for each $1\le i\le t+1$.
By \eqref{7-1} we get
\begin{equation}\label{7-2}
\syz^nk\text{ is $n$-torsionfree}.
\end{equation}
Fix a nonmaximal prime ideal $\p$ of $R$.
By assumption, the local ring $R_\p$ is Cohen--Macaulay and has minimal multiplicity.
If $R_\p$ is regular, then $\syz^{\height\p}\kappa(\p)$ is $R_\p$-free, and belongs to $\X_\p$.
Suppose that $R_\p$ is singular.
By assumption we get $\p\in \ipd(\tf_n(R))$, which implies $\p\in\ipd(G)$ for some $G\in\tf_n(R)$.
Then $\syz^{\height\p}G_\p$ is a nonfree maximal Cohen--Macaulay $R_\p$-module.
It is observed by using \cite[Proposition 5.2 and Lemma 5.4]{crs} and \cite[Lemma 3.2(1)]{crspd} that $\syz^{\height\p}\kappa(\p)\in\res_{R_\p}(\syz^{\height\p}G_\p)\subseteq\add_{R_\p}\X_\p$.
Thus we have
\begin{equation}\label{7-3}
\syz^{\height\p}\kappa(\p)\in\add_{R_\p}\X_\p\text{ for all nonmaximal prime ideals $\p$ of $R$}.
\end{equation}
It follows from \eqref{7-2} and \eqref{7-3} that the resolving subcategory $\X$ of $\mod R$ is dominant.

Fix an $R$-module $M\in\s_n(R)$.
The proof of the theorem will be completed once we prove that $M$ belongs to $\X$.
Fix a prime ideal $\p$ of $R$.
In view of \cite[Theorem 1.1]{dom}, it suffices to show that $\depth M_\p$ is not less than $r:=\inf_{X\in\X}\{\depth X_\p\}$.
As $M$ satisfies $(\ss_n)$, we have $\depth M_\p\ge\inf\{n,\height\p\}$.
If $\height\p\le n$, then $\depth M_\p\ge\height\p\ge r$, and we are done.
We may assume $\height\p>n$, and hence $\depth M_\p\ge n$ and $\depth R_\p\ge\inf\{n,\height\p\}=n$.
It is enough to deduce that $n\ge r$.

Consider the case where $\p=\m$.
In this case, we have $\depth R\ge n$.
Applying the depth lemma yields $\depth(\syz^nk)_\p=\depth\syz^nk=n$, while $\syz^nk\in\X$ by \eqref{7-2}.
Hence $n\ge r$.
Thus we may assume $\p\ne\m$.

The inequality $\height\p>n$ particularly says that $R_\p$ is not artinian, which implies that $\p\in\nf(R/\p)$.
By \cite[Lemma 4.6]{radius}, we find an $R$-module $C$ with $\nf(C)=\V(\p)$ and $\depth C_\q=\inf\{\depth R_\q,\depth(R/\p)_\q\}$ for all $\q\in\V(\p)$.
Set $Z=\syz^nC$.
As $\depth C_\p=0$ and $\depth R_\p\ge n$, the depth lemma says $\depth Z_\p=n$.
For each integer $1\le i\le n$ there are equalities and inequalities
\begin{align}
\label{7-7}\grade\Ext_R^i(C,R)&=\inf\{\depth R_\q\mid\q\in\supp\Ext_R^i(C,R)\}\\
\label{7-8}&\ge\inf\{\depth R,\depth R_\q\mid\m\ne\q\in\supp\Ext_R^i(C,R)\}\\
\label{7-9}&=\inf\{t,\height\q\mid\m\ne\q\in\supp\Ext_R^i(C,R)\}\\
\label{7-10}&\ge\inf\{t,\height\q\mid\m\ne\q\in\V(\p)\}
=\inf\{t,\height\p\}\ge n-1\ge i-1.
\end{align}
Here, \eqref{7-7} follows from \cite[Proposition 1.2.10(a)]{BH}, the inequality \eqref{7-8} is an equality unless $\Ext_R^i(C,R)=0$, and the equality \eqref{7-9} holds since $R$ is locally Cohen--Macaulay on the punctured spectrum.
In \eqref{7-10}, the first inequality holds since $\supp\Ext_R^i(C,R)\subseteq\nf(C)=\V(\p)$, the equality holds since $\p\ne\m$, and the second inequality follows from \eqref{7-1} and the fact that $\height\p>n$.
By \cite[Proposition (2.26)]{AB} the module $Z$ is $n$-torsionfree, and in particular, $Z\in\X$.
It is now seen that $n\ge r$.
\end{proof} 

To show our next proposition, we establish the lemma below, which is of independent interest.

\begin{lem}\label{a}
Let $(R,\m,k)$ be a $d$-dimensional Cohen--Macaulay non-Gorenstein local ring with minimal multiplicity.
Then $\G_{i,j}=\add R$ for all $i,j\ge0$ with $i+j\ge 2d+2$.
\end{lem}   

\begin{proof}
We freely use \cite[Proposition 1.1.1]{I}.
We may assume $i+j=2d+2$ since $\G_{a,b}$ contains $\G_{a+1,b}$ and $\G_{a,b+1}$ for all $a,b\ge0$.
As the stable category of $\G_{i,j}$ is equivalent to that of $\G_{2d+2,0}$, it suffices to show that every $M\in\G_{2d+2,0}$ is free.
Taking the faithfully flat map $R\to R[X]_{\m R[X]}$, we may assume that $k$ is infinite.
Choose an $R$-sequence $\xx=x_1,...,x_d$ with $\m^2=\xx\m$ by \cite[Exercise 4.6.14]{BH}.
We have $N:=\syz^dM\in\G_{d+2,d}\subseteq\cm(R)$.
Using the exact sequences $\{0\to N/\xx_{i-1}N\xrightarrow{x_i}N/\xx_{i-1}N\to N/\xx_iN\to0\}_{i=1}^d$ where $\xx_i=x_1,\dots,x_i$, we see that $\Ext_R^j(\overline N,R)=0$ for $j=d+1,d+2$ where $\overline{(-)}=(-)\otimes_RR/(\xx)$.
Hence $\Ext_{\overline R}^j(\overline N,\overline R)=0$ for $j=1,2$ by \cite[Lemma 3.1.16]{BH}.
In particular, $\Ext_{\overline R}^1(L,\overline R)=0$ where $L=\syz_{\overline R}\overline N$.
As $(\m\overline R)^2=0$, the module $L$ is a $k$-vecor space.
As $R$ is non-Gorenstein, we must have $L=0$, which means $\overline N$ is $\overline R$-free, which means $N$ is $R$-free (by \cite[Lemma 1.3.5]{BH}), which means $M$ is $R$-free (as $M\approx\syz^{-d}N$).
\end{proof}   

Using the above lemma, we can prove the following proposition.

\begin{prop}\label{b}
Suppose that for all minimal prime ideals $\p$ of $R$ the artinian local ring $R_{\p}$ has minimal multiplicity.
If $\ng (R) \subseteq \nf(\refl R)$ (e.g., if $\sing R\subseteq\nf(\refl R)$), then $R$ is generically Gorenstein.
\end{prop}

\begin{proof}
Take any $\p\in\Min R$.
By assumption, we have $(\p R_\p)^2=0$.
If $\p \in \ng(R)$, then Lemma \ref{a} implies $(\refl R)_\p\subseteq\refl(R_\p)=\add(R_\p)$.
But then $\p \notin \nf(\refl R)$, contradicting our assumption.
Thus if $\p \in \Min(R)$, then $\p \notin \ng(R)$, that is, $R_\p$ is Gorenstein.
\end{proof}

The corollary below, which is a consequence of the above proposition, gives kind of a converse to Proposition \ref{3-2}(4), and shows that in some cases the minimal multiplicity condition in Theorem \ref{tf-minimal}(1) actually forces some stringent condition on the Gorenstein locus of $R$, so that in those cases Theorem \ref{tf-minimal}(1) gives nothing newer than Proposition \ref{3-2}(4).

\begin{cor}\label{3-10}
Assume that $R$ satisfies $(\ss_2)$.
Consider the following four statements.\par
{\rm(1)} $R$ is generically a hypersurface.\quad
{\rm(2)} $R$ is generically Gorenstein.\par
{\rm(3)} $\s_2(R)=\res (\refl R)$.\quad
{\rm(4)} $\sing R=\ipd (\refl R)$.\\
Then $(1)\Rightarrow(2)\Rightarrow(3)\Rightarrow(4)$ hold.
If $R$ is generically of minimal multiplicity, then the four statements are equivalent.
If $R$ is Cohen--Macaulay and $\dim R\le2$, then {\rm(4)} is equivalent to $\sing R=\nf(\refl R)$.
\end{cor}

\begin{proof}
The last assertion is clear.
It is obvious that (1) implies (2).
Proposition \ref{3-2}(4) shows that (2) implies (3).
Since $\ipd(\s_2(R))=\sing R$, so (3) implies (4).

Suppose that $R$ is generically of minimal multiplicity, and assume (4).
Then $\sing R=\ipd(\refl R)\subseteq\nf(\refl R)$, and hence (2) holds by Proposition \ref{b}.
Finally, (2) implies (1) by the well-known (and easy to see) fact that an artinian Gorenstein local ring with minimal multiplicity is a hypersurface.
\end{proof} 

The first example below shows the non-vacuousness of Corollary \ref{3-10}.
The second says that one cannot drop the hypothesis of $R$ being generically of minimal multiplicity in the second part of Corollary \ref{3-10}.

\begin{ex}
Let $k$ be a field.
\begin{enumerate}[(1)]
\item
Let $R=k[\![x,y,z]\!]/(x^2,xy,y^2z)$.
Then $R$ is a $1$-dimensional Cohen--Macaulay local ring with $y-z$ a parameter.
Put $\p=(x,y)$, $\q=(x,z)$ and $\m=(x,y,z)$.
It holds that $R_\p\cong k[\![x,y,z]\!]_{(x,y)}/(x^2,xy,y^2)$ with $z$ a unit, while $R_\q\cong k[\![y]\!]_{(0)}$ with $y$ a unit.
We have $\ell\ell(R_\p)=2$, $\ell\ell(R_\q)=1$, $\spec R=\{\p,\q,\m\}$, $\Min R=\{\p,\q\}$ and $\sing R=\ng R=\{\p,\m\}$. Hence $R$ is generically of minimal multiplicity, $R$ is not generically Gorenstein and $\p$ does not belong to $\nf(\refl R)$. 
\item
Let $R=k[\![x,y,z,w]\!]/(x^2,y^2,yz,z^2w)$.
Then $R$ is a $1$-dimensional Cohen--Macaulay local ring with $w-z$ a parameter.
Set $\p=(x,y,z)$ and $\q=(x,y,w)$.
We have $R_\p\cong k[\![x,y,z,w]\!]_{(x,y,z)}/(x^2,y^2,yz,z^2)$ with $w$ a unit, while $R_\q\cong k[\![x,z]\!]_{(x)}/(x^2)$ with $z$ a unit.
It follows that $\spec R=\sing R=\{\p,\q,\m\}$, $\Min R=\{\p,\q\}$ and $\{\p,\m\}=\ng R$.
So $R$ is not generically Gorenstein.
As $(x,x)$ is an exact pair of zerodivisors (i.e., $0:x=(x)$), $R/(x)$ is a totally reflexive $R$-module, and in particular, it is reflexive.
We see that $\ng R\subseteq\sing R\subseteq\nf(R/(x))\subseteq\nf(\gp (R))\subseteq\nf(\refl R)\subseteq \nf (\cm(R))=\sing (R)$.
Note that $\ell\ell(R_\p)=3$, so that $R_\p$ does not have minimal multiplicity.
\end{enumerate}
\end{ex} 

In view of Corollary \ref{3-10}, we raise natural questions.

\begin{ques}
\begin{enumerate}[(1)]
\item
For a local Cohen--Macaulay ring $R$ (generically of minimal multiplicity) of dimension $d>1$, does the equality $\res(\tf_{d+1}(R))=\cm(R)$ force any condition on the Gorenstein locus of $R$\,?
\item
Let $R$ be an artinian local ring.
When does the equality $\ext (\refl R)=\mod R$ imply the Gorensteinness of $R$\,?
(Lemma \ref{a} gives one sufficient condition that $R$ has minimal multiplicity.)
\end{enumerate}
\end{ques}

\section{Syzygies of the residue field and direct summands of syzygies}

The proof of Theorem \ref{tf-minimal}(1) crucially uses the torsionfree nature of syzygies of the residue field of a local ring.
So, in this section, we record some results explaining when certain syzygies of the residue field of a local ring is or is not $n$-torsionfree for certain $n$ depending on the depth of the ring.
We begin with the following result.

\begin{thm}\label{syz}
Let $(R,\m,k)$ be a local ring of depth $t$.
Then the following statements hold.
\begin{enumerate}[\rm(1)]  
\item
Let $n\ge0$ be an integer, and put $m=\min \{n,t+1\}$.
Then the inclusion $\Syz_n(R)\cap\mod_0R\subseteq\tf_m(R)$ holds true.
In particular, the module $\syz^nk$ is $m$-torsionfree.
\item
The module $\syz^t k$ is $(t+1)$-torsionfree.   
\item
If $\syz^n k$ is $(t+2)$-torsionfree for some $n\ge t+1$, then $R$ is Gorenstein.
The converse is also true.
\end{enumerate}
\end{thm} 

\begin{proof} 
(1) Let $M$ be an $R$-module in $\Syz_n(R)\cap\mod_0R$.
Since $m\le n$, we have $M\in\Syz_n(R)\subseteq\Syz_m(R)$.
Let $\p$ be a prime ideal of $R$ such that $\depth R_\p\le m-2$.
Then $\depth R_\p\le m-2\le t-1$, which forces $\p\ne\m$.
Hence the $R_\p$-module $M_\p$ is free.
Applying \cite[Theorem 43]{M}, we obtain $M\in\tf_m(R)$.

(2) There is an $R$-sequence $\xx=x_1,\dots,x_t$.
Since $R/(\xx)$ has depth $0$, there is an exact sequence $0 \to k \to R/(\xx) \to C \to 0$.
Applying $\syz^t=\syz_R^t$ and remembering $\pd_RR/(\xx)=t$, we get an exact sequence $0 \to \syz^tk \to F \to \syz^tC \to 0$ of $R$-modules with $F$ free.
This shows $\syz^tk\approx\syz^{t+1}C$.
It follows from (1) that $\syz^{t+1}C\in\Syz_{t+1}(R)\cap\mod_0R\subseteq\tf_{t+1}(R)$, and therefore $\syz^tk\in\tf_{t+1}(R)$.

(3) First we prove the $n=t+1$ case.
With the notation of the proof of (1), we have $\syz^tk\approx\syz^{t+1}C$.
This implies $\syz^{t+1}k \approx \syz^{t+2}C$ and $\syz^{t+2}C\in\tf_{t+2}(R)$.
It follows from \cite[Corollary (4.18) and Proposition (2.26)]{AB} that $\Ext^{t+2}_R(C,R)$ has grade at least $t+1$.
Note that any nonzero $R$-module has grade at most $t=\depth R$.
We thus have $0=\Ext^{t+2}_R(C,R)=\Ext^1_R(\syz^{t+1}C,R)=\Ext^1_R(\syz^t k, R)=\Ext^{t+1}_R(k,R)$.
By \cite[II. Theorem 2]{roberts} we get $\id_RR<t+1$.
We conclude that $R$ is Gorenstein.

Next assume $n\ge t+2$.
We have $\syz^n k=\syz^{t+2} (\syz^{n-t-2}k)\in\tf_{t+2}(R)$.
Again by \cite[Corollary (4.18) and Proposition (2.26)]{AB}, $\Ext^{t+2}_R(\syz^{n-t-2}k,R)$ has grade at least $t+1$.
We have $0=\Ext^{t+2}_R(\syz^{n-t-2}k,R)=\Ext^n_R(k,R)$, and again by \cite[II. Theorem 2]{roberts} we get $\id_RR<n$ and $R$ is Gorenstein.
\end{proof}

\begin{rem}
Theorem \ref{syz}(2) vastly generalizes \cite[Proposition 4.1]{faber} in that we neither use any Cohen--Macaulay assumption on the ring, nor do we have any restriction on the depth of the ring. 
\end{rem}

As an immediate consequence of Theorem \ref{syz}, we can describe, for certain values of $n$, the resolving closure of $n$-torsionfree modules which are also locally free on the punctured spectrum.

\begin{cor}\label{loc-free}
Let $(R,\m,k)$ be a local ring of depth $t$.
Then there are equalities
\begin{align*}
\res(\tf_n(R)\cap \mod_0 R)&=\{M\in \mod_0 R\mid\depth M\ge n\}\text{ for $0\le n\le t$, and}\\
\res(\tf_{t+1}(R)\cap \mod_0 R)&=\{M\in \mod_0 R\mid\depth M\ge t\}.
\end{align*}
\end{cor} 

\begin{proof}
Fix $0\le n\le t$ and put $\X_n=\{M\in \mod_0 R\mid\depth M\ge n\}$.
That $\X_n$ is resolving is seen by the depth lemma etc.
Also, $\tf_n(R)\subseteq\Syz_n(R)\subseteq\X_n$ by the depth lemma again.
Therefore, $\res(\tf_n(R)\cap \mod_0 R)\subseteq\X_n$.
The reverse inclusion follows from \cite[Proposition 3.4]{dom} and Theorem \ref{syz}(1): we have $M\in \res\syz^n k\subseteq \res(\tf_n(R)\cap \mod_0 R)$ whenever $M\in \X_n$.
We obtain $\X_n=\res(\tf_n(R)\cap \mod_0 R)$ for every $0\le n\le t$.
For the other equality, we have $\res(\tf_{t+1}(R)\cap\mod_0R)\subseteq\res(\tf_t(R)\cap\mod_0R)=\X_t$, and the opposite inclusion follows similarly as above by using \cite[Proposition 3.4]{dom} and Theorem \ref{syz}(2).
\end{proof}

In view of Theorem  \ref{syz}(2) it would also be natural to ask if for any special classes of non-Gorenstein local Cohen--Macaulay rings $(R,\m, k)$ of dimension $d$, one can prove $\syz^d k\in \tf_n(R)$ or even $\Omega^d k\in \Syz_n(R)$ for some $n>d+1$.
We shall show that the answer is no.
For this, we record the following lemma, whose second assertion would be of independent interest.
Recall that for a local ring $(R,\m,k)$ of depth $t$ the number $\r(R)=\dim_k\Ext_R^t(k,R)$ is called the {\em type} of $R$.

\begin{lem}\label{4-2}
Let $(R,\m,k)$ be a local ring of depth $t$ and type $r$.
\begin{enumerate}[\rm(1)]
\item
Assume $t=0$.
Suppose that there is an exact sequence $0\to M\to F_1\xrightarrow{\partial}F_0$ of $R$-modules with $F_0,F_1$ free, $\image\partial\subseteq\m F_0$, $M=\bigoplus_{i=0}^n(\syz^ik)^{\oplus m_i}$ and $m_0=1$.
It then holds that $r=1$.
\item
The $R$-module $(\syz^t k)^{\oplus r}$ is $(t+2)$-syzygy.
\end{enumerate}
\end{lem}

\begin{proof}
(1) We may assume that $R$ is not a field.
Applying $\Hom_R(k,-)$ to the exact sequence and noting $\Hom_R(k,\partial)=0$, we get an isomorphism $\Hom_R(k,M)\cong\Hom_R(k,F_1)$.
Setting $s=\rank_RF_1$, we get
$$\textstyle
k^{\oplus rs}
\cong\Hom_R(k,F_1)
\cong\Hom_R(k,M)
\cong\Hom_R(k,\bigoplus_{i=0}^n(\syz^ik)^{\oplus m_i})
\cong\bigoplus_{i=0}^n\Hom_R(k,\syz^ik)^{\oplus m_i}.
$$
Hence $rs=\sum_{i=0}^nu_im_i$, where $u_i:=\dim_k\Hom_R(k,\syz^ik)$.
Note that for each $2\le i\le n$ there is an exact sequence $0\to\syz^ik\to R^{\oplus b_{i-1}}\xrightarrow{\delta_{i-1}}R^{\oplus b_{i-2}}$ with $\image\delta_{i-1}\subseteq\m R^{\oplus b_{i-2}}$ and $b_j=\beta_j(k)$ for each $j$.
Similarly as above, we obtain isomorphisms $k^{\oplus u_i}
\cong\Hom_R(k,\syz^ik)
\cong\Hom_R(k,R^{\oplus b_{i-1}})
\cong k^{\oplus rb_{i-1}}$, which imply $u_i=rb_{i-1}$ for all $2\le i\le n$.
As $R$ is not a field, the map $\Hom_R(k,\m)\to\Hom_R(k,R)$ induced from the inclusion map $\m\to R$ is an isomorphism.
Hence $u_i=rb_{i-1}$ for all $1\le i\le n$.
We have $u_0=\dim_k\Hom_R(k,k)=1$, while $m_0=1$ by assumption.
We obtain $rs=\sum_{i=0}^nu_im_i=u_0m_0+\sum_{i=1}^nrb_{i-1}m_i=1+\sum_{i=1}^nrb_{i-1}m_i$, and get $r(s-\sum_{i=1}^nb_{i-1}m_i)=1$.
This forces us to have $r=1$.

(2) Take an $R$-sequence $\xx=x_1,\dots,x_t$.
Then the socle of $R/(\xx)$ is isomorphic to $k^{\oplus r}$.
Let $y_1,\dots,y_n$ be elements of $\m$ whose residue classes form a system of generators of $\m/(\xx)$.
Then we have an exact sequence $0\to k^{\oplus r}\to R/(\xx)\xrightarrow{m}(R/(\xx))^{\oplus n}\to L\to0$, where $m$ is given by the transpose of the matrix $(y_1,\dots,y_n)$.
Applying the functor $\syz^t=\syz_R^t$, we get an exact sequence $0\to(\syz^tk)^{\oplus r}\to P\to Q\to\syz^tL\to0$ of $R$-modules with $P,Q$ free.
Consequently, we obtain the containment $(\syz^tk)^{\oplus r}\in \Syz_{t+2}(R)$.
\end{proof}

Now we can prove the following theorem.

\begin{thm}\label{type 1}
Let $(R,\m,k)$ be a local ring of dimension $d$ and depth $t$.
\begin{enumerate}[\rm(1)]
\item
The module $\syz^tk$ is $(t+2)$-syzygy if and only if the local ring $R$ has type one.
\item
Suppose that the subcategory $\Syz_{t+2}(R)$ is closed under direct summands.
Then $R$ has type one.
\item
Suppose that $R$ is Cohen--Macaulay.
Then the ring $R$ is Gorenstein if and only if $\syz^dk$ is $(d+2)$-syzygy, if and only if $\Syz_{d+2}(R)$ is closed under direct summands.
\end{enumerate}
\end{thm}

\begin{proof}
Assertion (2) immediately follows from (1) and Lemma \ref{4-2}(2).
Assertion (3) is a direct consequence of (1) and (2).
The ``if'' part of (1) follows from Lemma \ref{4-2}(2).
To show the ``only if'' part, put $r=\r(R)$. 
By assumption, there is an exact sequence $0\to\syz^tk\to F_{t+1}\xrightarrow{\partial_{t+1}}F_t\xrightarrow{\partial_t}\cdots\xrightarrow{\partial_1}F_0\to M\to0$ with $F_i$ free for all $0\le i\le t+1$.
If $\syz^tk$ has a nonzero free summand, then $R$ is regular by \cite[Corollary 1.3]{D} and $r=1$.
We may assume that $\syz^tk$ has no nonzero free summand, and hence we may assume $\image\partial_i\subseteq\m F_{i-1}$ for all $1\le i\le t+1$.
Set $N=\image\partial_t$.
The depth lemma shows $\depth N\ge t$.
Choose a regular sequence $\xx=x_1,\dots,x_t$ on $R$ and $N$ with $x_i\in\m\setminus\m^2$ for all $1\le i\le t$.
Putting $\overline{(-)}=(-)\otimes_RR/(\xx)$ and applying \cite[Corollary 5.3]{syz2}, we have an isomorphism $\overline{\syz_R^tk}\cong\bigoplus_{i=0}^t(\syz_{\overline R}^ik)^{\oplus\binom{t}{i}}$, and an exact sequence $0\to\overline{\syz^tk}\to\overline{F_{t+1}}\xrightarrow{\overline{\partial_{t+1}}}\overline{F_t}\to\overline N\to0$ is induced.
We apply Lemma \ref{4-2}(1) to obtain $r=1$.
\end{proof}

Theorem \ref{type 1}(1) may lead us to wonder whether the condition that $\syz^tk$ is $(t+2)$-syzygy already implies that $R$ is Gorenstein.
The next example answers in the negative.

\begin{ex}
Let $R$ be a local ring with $R/(\zz)\cong k[\![x,y]\!]/(x^2,xy)$ for some $R$-sequence $\zz=z_1,\dots,z_t$.
Then $R$ has depth $t$ and type $1$, so $\syz^tk$ is $(t+2)$-syzygy by Theorem \ref{type 1}(1), but $R$ is not Gorenstein.
\end{ex}  

In view of Theorem \ref{type 1}, it is natural to ask the following question.

\begin{ques}\label{ques-summand}
Let $R$ be a local ring with residue field $k$, and let $n\ge\depth R$ be an integer.
Assume that $\syz^n k$ is $(n+2)$-syzygy (note that this assumption is satisfied if $\Syz_{n+2}(R)$ is closed under direct summands by Lemma \ref{4-2}(2)).
Does then $R$ have type one?
What if we also assume $R$ is Cohen--Macaulay?
\end{ques} 

Note that Theorem \ref{type 1}(1) exactly says that Question \ref{ques-summand} has an affirmative answer when $n=t$, which is why we were able to derive Theorem \ref{type 1}(3).
We record some special cases, apart from that already contained in Theorem \ref{type 1}(3), where Question \ref{ques-summand} has a positive answer.

\begin{prop}\label{syz-ex}
Let $(R,\m,k)$ be a local ring of depth $t$.
Suppose that $\Syz_{n+2}(R)$ is closed under direct summands for some $n\ge t$.
Then $R$ has type one in either of the following two cases.
\begin{enumerate}[\rm(1)]
\item
One has $t>0$ and there is an $R$-sequence $\xx=x_1,\dots, x_{t-1}$ such that $\m/(\xx)$ is decomposable.
\item
The module $\syz^t k$ is a direct summand of $\syz^{t+l} k$ for some $l>0$ (this holds if $R$ is singular and Burch).
\end{enumerate}
\end{prop}

\begin{proof}
(1) Write $\overline R=R/(\xx)$ and $\overline \m=\m/(\xx)$.
Since $\overline \m$ is decomposable, $\overline R$ is singular.
Fix $s\ge0$.
By \cite[Theorem A]{NT} we have $\syz_{\overline R}k=\overline\m\lessdot\syz_{\overline R}^3L\oplus\syz_{\overline R}^4L\oplus\syz_{\overline R}^5L=\syz^{s+3}_{\overline R}N$, where $L=\syz_{\overline R}^sk$ and $N=k\oplus\syz_{\overline R}k\oplus\syz^2_{\overline R}k$.
Taking the functor $\syz_R^{t-1}$ and applying \cite[Lemma 4.2]{NT}, we obtain $\syz_R^tk\oplus F\lessdot\syz_R^{s+t+2}N\oplus G\in\Syz_{s+t+2}(R)$ for some free $R$-modules $F,G$.
By assumption, we can choose $s\ge0$ so that $\Syz_{s+t+2}(R)$ is closed under direct summands.
Then we have $\syz_R^t k\in\Syz_{s+t+2}(R)\subseteq\Syz_{t+2}(R)$, and Theorem \ref{type 1}(2) implies $\r(R)=1$.

(2) Applying the functor $\syz^l$ to the relation $\syz^tk\lessdot\syz^{t+l}k$ repeatedly, we have $\syz^t k\lessdot\syz^{t+bl}k$ for every $b\ge 1$.
Choose $b$ so that $t+bl\ge n+2$.
As $\Syz_{n+2}(R)$ is closed under direct summands, we get $\syz^t k\in\Syz_{n+2}(R)$.
Theorem \ref{type 1} implies $\r(R)=1$.
If $R$ is a singular Burch ring, then $\syz^t k\lessdot\syz^{t+2} k$ by \cite[Proposition 5.10]{burch}.
\end{proof}

Next we show that the converse to Theorem \ref{type 1}(2) is not true, that is, it is possible that $R$ is a local ring of depth $t$ and type $1$ but $\Syz_{t+2}(R)$ is not closed under direct summands.
Moreover, we shall show that for large classes of local rings $R$ with decomposable maximal ideal and of integers $n$ the subcategory $\Syz_n(R)$ is not closed under direct summands.
For this, we begin with establishing a lemma.

\begin{lem}\label{depth}
Let $(R,\m,k)$ be local and with $\depth R=t$.
Then $\depth M=t$ for each $0\ne M\in\Syz_{t+2}(R)$.
\end{lem}

\begin{proof}
By the depth lemma it suffices to show $\depth M\le t$, which holds if $R\lessdot M$.
We may let $M\cong\syz^{t+2}C$ for some $R$-module $C$, and get an exact sequence $0\to M\to F\xrightarrow{f} G\to N\to 0$ with $F,G$ free, $N$ $t$-syzygy and $\image f\subseteq\m G$.
Again by the depth lemma $\depth N\ge t$.
Break the exact sequence down as $0\to M\to F\to\syz N\to 0$ and $0\to \syz N\to G\to N\to 0$.
We get exact sequences $\Ext_R^t(k,M)\xrightarrow{a}\Ext_R^t(k,F)\xrightarrow{b}\Ext_R^t(k,\syz N)$ and $0=\Ext_R^{t-1}(k,N)\to\Ext_R^t(k,\syz N)\xrightarrow{c}\Ext_R^t(k,G)$.
Note that $cb=\Ext_R^t(k,f)=0$.
As $c$ is injective, we have $b=0$ and $a$ is surjective.
Since $\Ext_R^t(k,F)\ne0$, we obtain $\Ext_R^t(k,M)\ne0$.
\end{proof}  

\begin{rem}
The $(t+2)$nd threshold in Lemma \ref{depth} is sharp.
Indeed, let $(R,\m)$ be a local ring with $\dim R>0=\depth R$.
Then there exists $\m\ne\p\in\ass R$.
We have $\depth R/\p>0$ and $0\ne R/\p\in\Syz_1(R)$.
\end{rem}

Now we produce the promised classes of local rings $R$ and integers $n$.

\begin{prop}\label{depthzero}
Let $(R,\m,k)$ be a local ring such that $\m=I\oplus J$ for some nonzero ideals $I,J$ of $R$.
\begin{enumerate}[\rm(1)]
\item
Suppose that $\depth R/I=0$ and $\depth R/J\ge 1$.
Then one has $\depth R=0$, and the subcategory $\Syz_n(R)$ is not closed under direct summands for every $n\ge 2$.
\item
Suppose that $I$ is indecomposable and $\depth R/I \ge 2$.
Then one has $\depth R\le1$, and the subcategory $\Syz_n(R)$ is not closed under direct summands for every $n\ge4$.
\end{enumerate}
\end{prop}  

\begin{proof}
By \cite[Fact 3.1]{NT} we have $\depth R=\min\{\depth R/I,\depth R/J,1\}\le1$.

(1) Since $I\cong \m/J\subseteq R/J$ and $\depth R/J>0$, we see that $\depth I\ge 1$.
Lemma \ref{depth} yields $I\notin\Syz_2(R)$.
By \cite[Theorem A]{NT} we get $I\lessdot\m\lessdot\syz^3(\syz^ik)\oplus\syz^4(\syz^ik)\oplus\syz^5(\syz^ik)\in \Syz_{i+3}(R)\subseteq\Syz_{i+2}(R)$ for every $i\ge 0$.
If $\Syz_n(R)$ is closed under direct summands for some $n\ge 2$, then choosing $i=n-2$, we obtain $I\in\Syz_n(R)\subseteq\Syz_2(R)$, which is a contradiction.

(2) Note that $IJ=0$.
We also have $J^2\ne 0$, as otherwise $\m^2=(I+J)(I+J)=I^2\subseteq I$, contradicting $\depth R/I\ge 2$.
Similarly as in the proof of (1), for each $i\ge0$ there exists $X\in\Syz_{i+3}(R)$ with $I\lessdot X$.

Assuming that $\Syz_n(R)$ is closed under direct summands for some $n\ge4$, we shall derive a contradiction.
Choosing $i=n-3$, we get $I\in \Syz_n(R)\subseteq \Syz_4(R)$.
The equality $IJ=0$ implies that $I$ does not have a nonzero free summand.
Hence $I\cong\syz^4H$ for some $R$-module $H$.
Putting $M=\syz^3H$, we get an exact sequence $0\to I\to R^{\oplus a}\to M\to 0$.
Note that $M\ne 0$.
By \cite[Proposition 4.2]{moore} there are an $R/I$-module $A$ and an $R/J$-module $B$ such that $M\cong A \oplus B$.
Now $I=\syz_R M\cong \syz_RA \oplus \syz_R B$.
Then indecopmosability of $I$ implies $\syz_R A=0$ or $\syz_RB=0$, hence $A$ or $B$ is $R$-free.
As $IA=0=JB$, either $A=0$ or $B=0$.
If $A=0$, then we get an exact sequence $0\to I\to R^{\oplus a}\to B\to 0$ and $I,B$ are annihilated by $J$, whence $J^2R^{\oplus a}=0$, contradicting $J^2\ne0$.
We get $B=0$ and an exact sequence $0\to I\to R^{\oplus a}\to A\to 0$.
As $IA=0$, the surjection $R^{\oplus a}\to A$ factors through the canonical surjection $R^{\oplus a}\to (R/I)^{\oplus a}$, which induces a surjection $I\to \syz_{R/I}A$.
Since $IJ=0$, the module $\syz_{R/I}A$ is annihilated by $I,J$ and by $I+J=\m$.
Thus $\syz_{R/I}A\cong k^{\oplus s}$ for some $s\ge 0$.
But $\syz_{R/I} A$ embeds inside a free $R/I$-module which has positive depth, hence $k$ cannot be a summand of $\syz_{R/I} A$, thus $\syz_{R/I}A=0$.
Therefore $A$ is $R/I$-free and has depth at least $2$.
But $A\cong M$ has depth at most $1$ by Lemma \ref{depth}.
We now have a desired contradiction.
\end{proof}

The ring $R$ in the first example below shows that the converse to Theorem \ref{type 1}(2) does not hold.
The second example presents a local ring $R$ of depth $1$ and type $1$ such that $\Syz_n(R)$ is not closed under direct summands for every $n\ge4$, which concretely illustrates Proposition \ref{depthzero}(2).
The third example shows that the assumption $\depth R/I\ge 2$ in Proposition \ref{depthzero}(2) cannot be dropped.

\begin{ex}
Let $k$ be a field.
In each of the following statements, $\m$ denotes the maximal ideal of $R$.
\begin{enumerate}[(1)]
\item
Consider the local ring $R=k[\![x,y]\!]/(x^2,xy)$.
Then $\m=(y) \oplus (x)$, $\depth R/(y)=0$ and $\depth R/(x)=1$.
Proposition \ref{depthzero}(1) shows that $\Syz_n(R)$ is not closed under direct summands for all $n\ge 2$.
\item
Let $R=k[\![x,y,z]\!]/(xy,xz)$.
Then $\m=(x)\oplus(y,z)$ and $(x)$ is indecomposable with $\depth R/(x)=2$.
Proposition \ref{depthzero}(2) implies that $\Syz_{n+3}(R)$ is not closed under direct summands for every $n\ge4$.
\item
Let $R=k[\![x,y]\!]/(xy)$.
Then $\m=(x)\oplus(y)$ and $(x)$ is indecomposable with $\depth R/(x)=1$.
As $R$ is a $1$-dimensional Gorenstein ring, $\Syz_n(R)=\cm(R)$ is closed under direct summands for any $n\ge1$.
\end{enumerate}
\end{ex}  

\section{Closedness under extensions and syzygies, and totally reflexive modules} 

In this section, we derive some consequences of $\tf_n (R)$ being closed under extensions or syzygies for certain values of $n$ depending on the depth of the local ring $R$.
We will see that in the most reasonable cases, if $\tf_n(R)$ is resolving, then it coincides with the category of totally reflexive modules.

We begin with investigating when the subcategory $\tf_n(R)$ is closed under extensions or syzygies.
For two subcategories $\X,\Y$ of $\mod R$ we denote by $\X \ast \Y$ the subcategory consisting of modules $X$ that fits into an exact sequence $0\to M\to X\to N\to 0$ with $M\in \X$ and $N\in \Y$.

\begin{prop}\label{tf-syz}
Let $n$ be either a nonnegative integer or $\infty$.
The following statements hold.
\begin{enumerate}[\rm(1)]
\item
If $M$ is an $R$-module such that $(\add R)\ast M\subseteq \tf_n(R)$, then there is a containment $\syz M\in \tf_{n+1}(R)$.
\item
If the subcategory $\tf_n(R)$ is closed under extensions, then the inclusion $\syz\tf_n(R)\subseteq\tf_{n+1}(R)$ holds.
\item
The subcategory $\tf_n(R)$ is resolving if (and only if) $\tf_n(R)$ is closed under extensions.
\item
Suppose that $\tf_{n+1}(R)$ is closed under extensions.
If $R$ satisfies $(\ss_n)$, then $R_\p$ is Gorenstein for all $\p\in\spec R$ with $\height\p=n$.
If $R$ satisfies $(\ss_{n+1})$ and is local with $\dim R\ge n$, then $R$ satisfies $(\g_n)$.
\item
Suppose that $R$ is a local ring of depth $t$.
\begin{enumerate}[\rm(a)]
\item
Let $n\ge t+2$.
Then $\tf_{n}(R)=\gp(R)$ if (and only if) $\tf_{n} (R)$ is closed under syzygies.
\item
Let $n\ge t+1$.
Then $\tf_n(R)=\gp(R)$ if (and only if) $\tf_{n}(R)$ is closed under extensions.
When $n=t+1$, it is also equivalent to the Gorenstein property of the local ring $R$.
\end{enumerate}
\end{enumerate}
\end{prop}

\begin{proof}
(1) By \cite[Proposition (2.21)]{AB} and its proof, there is an exact sequence $0\to P\to N\to M\to 0$ with $P\in\add R$, where $N=\tr\syz\tr\syz M$.
We have $N\in(\add R)\ast M\subseteq \tf_{n}(R)$, and hence $N\in\tf_n(R)\cap\G_{1,0}=\G_{1,n}$.
By \cite[Proposition 1.1.1]{I} and \cite[Theorem 2.17]{AB} we get $\syz M\approx\syz N\in \syz  \G_{1,n}(R) \subseteq \tf_{n+1}(R)$.

(2) The assertion immediately follows from (1) as we have $(\add R)\ast\tf_n(R)\subseteq \tf_n(R)$ by assumption.

(3) The assertion is a direct consequence of (2) together with the general fact $\tf_{n+1}(R)\subseteq\tf_n(R)$.

(4) The second assertion follows by the first and \cite[Theorem 4.1]{syzext}.
To show the first assertion, let $\p\in\spec R$ with $\height\p=n$, $i\ge0$ and $M=\Ext_R^i(R/\p,R)$.
Each $\q\in\supp M$ contains $\p$, so $\height\q\ge\height\p=n$.
As $R$ satisfies $(\ss_n)$, we get $\depth R_\q\ge n$.
By \cite[Proposition 1.2.10(a)]{BH} we have $\grade M\ge n$.
It follows from \cite[Proposition (2.26)]{AB} that $\syz^{n+1}(R/\p)\in\tf_{n+1}(R)$.
We see from (2) that $\syz\tf_{n+1}(R)\subseteq\tf_{n+2}(R)$, and hence $\syz^{n+2}(R/\p)\in\tf_{n+2}(R)$, which induces $\syz^{n+2}\kappa(\p)\in\tf_{n+2}(R_\p)$.
Using again the assumption that $R$ satisfies $(\ss_n)$, we have $\depth R_\p=n$.
It follows from Theorem \ref{syz}(3) that $R_\p$ is Gorenstein.

(5a) Fix $M\in\tf_{n}(R)$.
By assumption $\syz M\in\tf_{n}(R)$ i.e. $\Ext_R^i(\tr\syz M,R)=0$ for all $1\le i\le n$
By \cite[Theorem (2.8)]{AB} there is an exact sequence $0=\Tor_1^R(M,R)\to(\Ext_R^1(M,R))^\ast\to\Ext_R^2(\tr\syz M,R)=0$.
Hence $(\Ext_R^1(M,R))^\ast=0$.
By \cite[Proposition (2.6)]{AB}, there exists an exact sequence $0\to\Ext_R^1(M,R)\to\tr M\to\syz\tr\syz M\to0$, which induces an exact sequence $\Ext_R^i(\tr M,R)\to\Ext_R^i(\Ext_R^1(M,R),R)\to \Ext^{i+1}_R(\syz \tr \syz M,R)=\Ext_R^{i+2}(\tr\syz M,R)$ for every $i\ge0$.
This implies $\Ext_R^i(\Ext_R^1(M,R),R)=0$ for all $1\le i\le n-2$.
Thus $\Ext_R^i(\Ext_R^1(M,R),R)=0$ for all $0\le i\le n-2$, that is, $\grade\Ext_R^1(M,R)\ge n-1>t$.
As $\depth R=t$, we must have $\Ext^1_R(M,R)=0$.
Replacing $M$ by $\syz^iM$ for $i\ge0$, we get $\Ext_R^{>0}(M,R)=0$ i.e. $M\in\G_{\infty,0}$.
So, $\tr M\in\G_{0,\infty}\subseteq\tf_{n}(R)$.
By what we have seen, $\tr M\in\G_{\infty,0}$.
Thus $M\in\gp(R)$.

(5b) Using (2), we observe that $\syz\tf_n(R)\subseteq\tf_{n+1}(R)\subseteq\tf_n(R)$.
The case $n \ge t+2$ follows from (5a).
Now we consider the case $n=t+1$.
Theorem \ref{syz}(1) yields that $\syz^nk\in \tf_n(R)$, which implies that $\syz^{n+1}k\in\syz\tf_n(R)\subseteq\tf_{n+1}(R)$.
Theorem \ref{syz}(3) implies that $R$ is Gorenstein and thus $\tf_n(R)=\gp(R)$.
Therefore, $R$ is a Gorenstein ring if (and only if) the subcategory $\tf_{t+1}(R)$ is closed under extensions.
\end{proof} 

\begin{rem}
\begin{enumerate}[(1)]
\item
Proposition \ref{tf-syz}(1) for $n=1$ recovers \cite[Corollary 2.2]{ho}.
Indeed, it says if $\tf_{1}(R)=\Syz_1(R)$ is closed under extensions, then $\syz M \in \tf_{2}(R)=\refl R$ for $M\in \tf_{1}(R)$.
As $\syz^2 \tr M\approx M^\ast$, we have that if $\Syz_1(R)$ is closed under extensions, then $M^\ast\in\refl R$ for every $M\in\mod R$.
\item
The inequality $\dim R\ge n$ in Proposition \ref{tf-syz}(4) is sharp.
Indeed, if $(R,\m)$ is a non-Gorenstein local ring with $\m^2=0$, then $\tf_2(R)=\proj R$ (see Lemma \ref{a}) is closed under extensions.
\item
Proposition \ref{tf-syz}(4) refines \cite[Theorem 2.3(3)$\Rightarrow$(2)]{MTT} in some cases.
\end{enumerate}
\end{rem}

Next we show the proposition below, which can also be of some independent interest. 
  
\begin{prop}\label{ext-local}
Let $(R,\m,k)$ be a local ring of depth $t$.
Let $a\in \{t,t+1 \}$ and $n\ge a$.
\begin{enumerate}[\rm(1)]
\item
Let $K\in\mod_0R$.
If $\Ext_R^{a+1}(\tr M,R)=0$ for all $M\in R\ast \syz^nK$, then $\Ext_R^{n+1}(K,R)=0$.
\item
If $\Ext_R^{a+1}(\tr M,R)=0$ for every $M\in R\ast \syz^n k$, then the local ring $R$ is Gorenstein. 
\end{enumerate}
\end{prop} 

\begin{proof}
(1) As $n\ge t$, the conclusion is clear if $R$ is Gorenstein.
Let $R$ be non-Gorenstein.
Put $L=\syz^nK$.
Assuming $\Ext^{n+1}_R(K,R)=\Ext^1_R(L,R)\ne 0$, we work towards a contradiction.
The choice of $K$ implies that $\Ext^1_R(L,R)$ has finite length.
Choose a socle element $0\ne\sigma\in\Ext^1_R(L,R)$.
Let $0\to R\to N\to L\to 0$ be the exact sequence corresponding to $\sigma$.
Then $N$ is in $R\ast L$, and $\Ext_R^{a+1}(\tr N,R)=0$ by assumption.
Using the snake lemma, we get the following commutative diagrams with exact rows and columns.
$$
\xymatrix@R-1pc@C-1pc{
0\ar[r]&0\ar[r]\ar[d]&R^{\oplus b}\ar[r]\ar[d]&R^{\oplus b}\ar[r]\ar[d]&0\\
0\ar[r]&R\ar[r]\ar[d]&R^{\oplus c+1}\ar[r]\ar[d]&R^{\oplus c}\ar[r]\ar[d]&0\\
0\ar[r]&R\ar[r]\ar[d]&N\ar[r]\ar[d]&L\ar[r]\ar[d]&0\\
&0&0&0
}\quad
\xymatrix@R-1pc@C-1pc{
&0\ar[d]&0\ar[d]&0\ar[d]\\
0\ar[r]& L^\ast\ar[r]\ar[d]& N^\ast\ar[r]^f\ar[d]& R\ar[d]\ar[llddd]\\
0\ar[r]& R^{\oplus c}\ar[r]\ar[d]& R^{\oplus c+1}\ar[r]\ar[d]& R\ar[r]\ar[d]& 0\\
0\ar[r]& R^{\oplus b}\ar[r]\ar[d]& R^{\oplus b}\ar[r]\ar[d]& 0\ar[r]\ar[d]& 0\\
&A\ar[r]\ar[d]& B\ar[r]\ar[d]& 0\ar[r]\ar[d]& 0\\
&0&0&0
}\quad
\xymatrix@R-1pc@C-1pc{
& 0\ar[d]& 0\ar[d]& 0\ar[d]\\
0\ar[r]& 0\ar[r]\ar[d]& B^\ast\ar[r]^l\ar[d]& A^\ast\ar[d]\ar[llddd]^(.4)\delta\\
0\ar[r]& 0\ar[r]\ar[d]& R^{\oplus b}\ar[r]\ar[d]& R^{\oplus b}\ar[r]\ar[d]& 0\\
0\ar[r]& R\ar[r]\ar[d]& R^{\oplus c+1}\ar[r]\ar[d]& R^{\oplus c}\ar[r]\ar[d]& 0\\
& R\ar[r]_h\ar[d]& N\ar[r]\ar[d]& L\ar[r]\ar[d]& 0\\
&0&0&0
}
$$
Note that $A\approx\tr\syz^nK$ and $B\approx\tr N$.
There is also an exact sequence $N^\ast\xrightarrow{f}R\xrightarrow{g}\Ext_R^1(\syz^nK,R)$ with $g(1)=\sigma$.
It is seen that $\cokernel f\cong k$.
We get exact sequences $0\to k\to A\to B\to0$, and
\begin{equation}\label{r2}
\Ext_R^a(B,R)\to\Ext_R^a(A,R)\to\Ext_R^a(k,R)\to\Ext_R^{a+1}(B,R)=\Ext_R^{a+1}(\tr N,R)=0.
\end{equation}
Using Theorem \ref{syz}(1) and the general fact that $\tf_{i+1}(R)\subseteq\tf_i(R)$ for $i\ge0$, we obtain $\syz^n K\in\tf_a(R)$.
Suppose $a\ge1$.
Then $\Ext_R^a(A,R)=\Ext^a_R(\tr \syz^n K,R)=0$, and it follows from \eqref{r2} that $\Ext^a_R(k,R)=0$.
As $a\in\{t,t+1\}$, we must have $a=t+1$.
By \cite[II. Theorem 2]{roberts} the ring $R$ is Gorenstein, and we get a desired contradiction.
Thus we may assume $a=0$, and then $t=0$.
The slanted arrow $\delta:A^\ast\to R$ in the third diagram above is a zero map by the injectivity of the map $h$.
Hence the map $l$ is surjective (and hence an isomorphism).
From \eqref{r2} we get $k^\ast=0$, which is a contradiction to the fact that $t=0$.

(2) By (1) we have $\Ext_R^{n+1}(k,R)=0$.
Since $n\ge t$, the ring $R$ is Gorenstein by \cite[II. Theorem 2]{roberts}.
\end{proof}   

Applying the above proposition, we can prove the following theorem.

\begin{thm}\label{ext-local-thm}
Let $(R,\m,k)$ be local with depth $t$.
Then $R$ is Gorenstein if and only if $\Syz_m(R)$ is closed under extensions for some integer $m>t$, if and only if $R\ast \syz^n k\subseteq \Syz_{t+1}(R)$ for some integer $n\ge t$.
\end{thm} 
 
\begin{proof}
First, suppose that $\Syz_m(R)$ is closed under extensions for some integer $m>t$.
Then $m-1\ge t$ and $\syz^{m-1}k\in\Syz_m(R)$ by Theorem \ref{syz}(2).
We obtain $R\ast\syz^{m-1}k\in\Syz_m(R)\subseteq\Syz_{t+1}(R)$.
Next, suppose that $R\ast\syz^nk\subseteq\Syz_{t+1}(R)$ for some $n\ge t$.
Then $R\ast\syz^nk\subseteq\Syz_{t+1}(R)\cap\mod_0R\subseteq\tf_{t+1}(R)$ by Theorem \ref{syz}(1).
Hence $\Ext_R^{t+1}(\tr M,R)=0$ for all $M\in R\ast\syz^nk$.
Proposition \ref{ext-local}(2) implies $R$ is Gorenstein.
\end{proof}   

To show our next result, we prepare a lemma to get a certain property of $\widetilde\s_n(R)$.

\begin{lem}\label{r3}
If $\syz \widetilde \s_n(R)$  is closed under extensions, then one has the equality $\syz \widetilde \s_n(R)=\widetilde \s_{n+1}(R)$.
\end{lem}

\begin{proof}
We easy see that $\syz\widetilde\s_n(R)\subseteq\widetilde\s_{n+1}(R)$.
The subcategory $\widetilde \s_n(R)$ is resolving and contains $\Syz_n(R)$.
By assumption and by Lemma \ref{3-1}(2) the subcategory $\syz\widetilde\s_n(R)$ is resolving, and it contains $\syz\Syz_n(R)=\Syz_{n+1}(R)$.
Hence $\syz\widetilde\s_n(R)$ contains $\res\Syz_{n+1}(R)=\widetilde\s_{n+1}(R)$ by Proposition \ref{3-2}(1).
\end{proof}

The corollary below states several consequences of Theorem \ref{ext-local-thm}.
The first two assertions of the corollary give necessary and sufficient conditions for $R$ to be Gorenstein.
In \cite[Theorem 2.3(8)]{MTT} it is required that $R$ satisfies $(\ss_n)$ along with $\Syz_n(R)$ being closed under extensions.
In the third assertion of the corollary, we show some special cases where the condition of $(\ss_n)$ can be dropped.

\begin{cor}\label{gor-s_n}
Let $(R,\m,k)$ be a local ring of dimension $d$ and depth $t$.
\begin{enumerate}[\rm(1)]
\item
One has that $R$ is a Gorenstein local ring if and only if $\syz \widetilde \s_n(R)$ is closed under extensions for some consecutive $(t+1)$-many values of $n$.
\item
The ring $R$ is Gorenstein if and only if $R$ is Cohen--Macaulay and $\syz \cm(R)$ is closed under extensions.
\item
If $\Syz_n(R)$ is closed under extensions for some $n\ge \min \{d,t+1\}$, then $R$ is Cohen--Macaulay and $\tf_i(R)=\Syz_i(R)$ for all $1\le i\le n+1$.
\end{enumerate}
\end{cor}

\begin{proof}
(1) If $R$ is Gorenstein, then for all $n\ge d$ one has $\widetilde\s_n(R)=\cm(R)$ and $\syz\widetilde\s_n(R)=\cm(R)$, and $\cm(R)$ is closed under extensions.
This shows the ``only if'' part.
From now on we prove the ``if'' part.
There is an integer $l\ge0$ such that $\syz\widetilde\s_n(R)$ is closed under extensions for $l\le n\le l+t$.
We have $\syz^lk\in\widetilde\s_l(R)$, so that $\syz^{t+1+l}k\in\syz^{t+1}\widetilde\s_l(R)=\widetilde\s_{t+1+l}(R)$ by Lemma \ref{r3}.
Since $\widetilde\s_{t+1+l}(R)$ is resolving, this implies $R\ast\syz^{t+1+l}k\subseteq\widetilde\s_{t+1+l}(R)=\syz^{t+1}\widetilde\s_l(R)\subseteq\Syz_{t+1}(R)$.
Theorem \ref{ext-local-thm} implies that $R$ is Gorenstein.

(2) If $R$ is Cohen--Macaulay, then $\widetilde\s_n(R)=\cm(R)$ for all $n\ge d$.
If $R$ is Gorenstein, then $\syz\cm(R)=\cm(R)$.
The assertion now follows from (1).

(3) If $n\ge t+1$, then $R$ is Gorenstein by Theorem \ref{ext-local-thm} and $\tf_i(R)=\Syz_i(R)$ for all $i\ge0$ by \cite[Corollary (4.22)]{AB}.
Assume $n\ge d$ and $R$ is not Cohen--Macaulay.
Then $n\ge d\ge t+1$ and Theorem \ref{ext-local-thm} implies $R$ is Gorenstein, which contradicts the assumption that $R$ is not Cohen--Macaulay.
Thus $R$ must be Cohen--Macaulay, hence satisfies $(\ss_n)$.
By \cite[Theorem 2.3(8)$\Rightarrow$(6)]{MTT} we get $\tf_{n+1}(R)=\Syz_{n+1}(R)$, and finally by \cite[Corollary 4.18]{AB} we obtain $\tf_i(R)=\Syz_i(R)$ for all $1\le i\le n+1$.
\end{proof}

\begin{ques}
Let $R$ be a local ring of depth $t$ such that $\Syz_t(R)$ is closed under extensions.
Then, is $\tf_t(R)$ also closed under extensions, or at least closed under syzygies?
(Note that Proposition \ref{tf-syz}(3) implies that if $\tf_t(R)$ is closed under extensions, then it is closed under syzygies.)
\end{ques}

Here we record the following observation on $\G_{a,b}$.
We should compare it with Proposition \ref{tf-syz}(5a).

\begin{prop}\label{7-5}
\begin{enumerate}[\rm(1)]
\item
Let $0<a<\infty$ and $0\le b\le\infty$.
Then the subcategory $\G_{a,b}$ is closed under syzygies if and only if there is an equality $\G_{a,b}=\gp(R)$.
\item
Suppose that $(R,\m,k)$ is a local ring of depth $t$.
Let $n\in\ZZ_{>0}\cup\{\infty\}$ be such that $\G_{n,0}$ is closed under syzygies.
Then one has $n\ge t$.
If one also has $n\le t+1$, then $R$ is a Gorenstein local ring.
\end{enumerate}
\end{prop}  
 
\begin{proof}
(1) For the first assertion, it suffices to show the ``only if'' part.
Let $M\in\G_{a,b}$, and $i\ge0$ an integer.
By assumption, $\syz^iM$ is in $\G_{a,b}$.
As $a>0$, we have $\syz^iM\in\G_{1,0}$, that is, $\Ext_R^{i+1}(M,R)=\Ext_R^1(\syz^iM,R)=0$.
Hence $M\in\G_{\infty,0}$, and thus $M\in\G_{a,b}\cap\G_{\infty,0}=\G_{\infty,b}$.
It follows that $\G_{a,b}=\G_{\infty,b}$.
We are done for $b=\infty$, so assume $b<\infty$.
We get $\G_{b,a}=\tr\G_{a,b}=\tr\G_{\infty,b}=\G_{b,\infty}$ by \cite[Proposition 1.1.1]{I}.
Note that $b-a$ is an integer.
Hence $\syz^{b-a}$ is defined, and we obtain $\G_{a,b}=\syz^{b-a}\G_{b,a}=\syz^{b-a}\G_{b,\infty}=\G_{a,\infty+b-a}=\G_{a,\infty}$ by \cite[Proposition 1.1.1]{I} again.
It follows that $\G_{a,b}=\G_{a,\infty}$, and $\G_{a,b}=\G_{a,\infty}\cap\G_{\infty,b}=\G_{\infty,\infty}=\gp(R)$.

(2) If $n<t$, then $k$ is in $\G_{n,0}$ and so does $\syz^{t-n}k$ by assumption, which implies $\Ext_R^t(k,R)=0$, a contradiction.
Hence $n\ge t$.
Suppose $n\le t+1$.
As $\syz^{t+1}k\in\tf_{t+1}(R)$ by Theorem \ref{syz}(1), applying (1) shows $\tr \syz^{t+1}k\in\G_{t+1,0}\subseteq\G_{n,0}=\gp(R)$.
Thus $\syz^{t+1}k\in\gp(R)$, which implies that $R$ is Gorenstein.
\end{proof} 

\begin{cor}\label{5-8}
Let $n\ge0$ be an integer such that $\tf_{n}(R)=\tf_{n+1}(R)$.
Then $\tf_n(R)=\gp(R)$.
\end{cor}

\begin{proof}
Using \cite[Proposition 1.1.1]{I} shows $\syz\G_{n+1,0}=\G_{n,1}\subseteq\G_{n,0}=\tr\tf_n(R)=\tr\tr_{n+1}(R)=\G_{n+1,0}$.
Proposition \ref{7-5}(1) yields $\G_{n+1,0}=\gp(R)$.
We conclude $\tf_n(R)=\tf_{n+1}(R)=\tr\G_{n+1,0}=\gp(R)$.
\end{proof}

\begin{rem}
The dual version of Proposition \ref{7-5}(1) holds true as well:
Let $0\le a\le\infty$ and $0<b<\infty$.
Then $\G_{a,b}$ is closed under cosyzygies if and only if $\G_{a,b}=\gp(R)$.
This is a consequence of the combination of Proposition \ref{7-5}(1), \cite[Lemma 4.1]{cosyz} and \cite[Proposition 1.1.1]{I}.
\end{rem}


\end{document}